\documentclass[12pt]{amsart}
\usepackage{amscd,amsmath,amsthm,amssymb}
\usepackage{pstcol,pst-plot,pst-3d}%\usepackage[T1]{fontenc}
\usepackage{color}
\usepackage{pstricks}
\usepackage{lineno,xcolor}
\newpsstyle{fatline}{linewidth=1.5pt}
\newpsstyle{fyp}{fillstyle=solid,fillcolor=verylight}
\definecolor{verylight}{gray}{0.97}
\definecolor{light}{gray}{0.9}
\definecolor{medium}{gray}{0.85}
\definecolor{dark}{gray}{0.6}

 \def\G{{\mathcal G}}

 \def\opn#1#2{\def#1{\operatorname{#2}}} % to make operators
 \opn\chara{char} \opn\length{\ell} \opn\pd{pd} \opn\rk{rk}
 \opn\projdim{proj\,dim} \opn\injdim{inj\,dim} \opn\rank{rank}
 \opn\depth{depth} \opn\grade{grade} \opn\height{height}
 \opn\embdim{emb\,dim} \opn\codim{codim}
 
 \opn\Tr{Tr} \opn\bigrank{big\,rank}
 \opn\superheight{superheight}\opn\lcm{lcm}
 \opn\trdeg{tr\,deg}%\emph{
 \opn\reg{reg} \opn\lreg{lreg} \opn\ini{in} \opn\lpd{lpd}
 \opn\size{size} \opn\sdepth{sdepth}
 \opn\link{link}\opn\fdepth{fdepth}\opn\lex{lex}
 \opn\tr{tr}
 \opn\type{type}
 \opn\Borel{Borel}
 \opn\div{div} \opn\Div{Div} \opn\cl{cl} \opn\Cl{Cl}
 \opn\Spec{Spec} \opn\Supp{Supp} \opn\supp{supp} \opn\Sing{Sing}
 \opn\Ass{Ass} \opn\Min{Min}\opn\Mon{Mon}
 \opn\Ann{Ann} \opn\Rad{Rad} \opn\Soc{Soc}
 \opn\Im{Im} \opn\Ker{Ker} \opn\Coker{Coker} \opn\Am{Am}
 \opn\Hom{Hom} \opn\Tor{Tor} \opn\Ext{Ext} \opn\End{End}
 \opn\Aut{Aut} \opn\id{id}
 
 \opn\nat{nat}
 \opn\pff{pf}%   \pf exists already
 \opn\Pf{Pf} \opn\GL{GL} \opn\SL{SL} \opn\mod{mod} \opn\ord{ord}
 \opn\Gin{Gin} \opn\Hilb{Hilb}\opn\sort{sort}
 \opn\PF{PF}\opn\Ap{Ap}
 \opn\aff{aff} \opn
 \con{conv} \opn\relint{relint} \opn\st{st}
 \opn\lk{lk} \opn\cn{cn} \opn\core{core} \opn\vol{vol}  \opn\inp{inp} \opn\nilpot{nilpot}
 \opn\link{link} \opn\star{star}\opn\lex{lex}\opn\set{set}
 \opn\width{wd}
 \opn\Fr{F}
 \opn\QF{QF}
 \opn\G{G}
 \opn\type{type}\opn\res{res}
 \opn\gr{gr}
 
 \def\pot#1#2{#1[\kern-0.28ex[#2]\kern-0.28ex]}

 \opn\dirlim{\underrightarrow{\lim}}
 \opn\inivlim{\underleftarrow{\lim}}
 \def\Implies{\ifmmode\Longrightarrow \else
         \unskip${}\Longrightarrow{}$\ignorespaces\fi}
 \def\implies{\ifmmode\Rightarrow \else
         \unskip${}\Rightarrow{}$\ignorespaces\fi}
 \def\iff{\ifmmode\Longleftrightarrow \else
         \unskip${}\Longleftrightarrow{}$\ignorespaces\fi}

 \let\:=\colon
 \newtheorem{Theorem}{Theorem}[section]
 \newtheorem{Lemma}[Theorem]{Lemma}
 \newtheorem{Corollary}[Theorem]{Corollary}

 \newtheorem{Example}[Theorem]{Example}
 
 \newtheorem{Definition}[Theorem]{Definition}

 \let\epsilon\varepsilon
 \let\kappa=\varkappa
 \def\qed{\ifhmode\textqed\fi
       \ifmmode\ifinner\quad\qedsymbol\else\dispqed\fi\fi}
 \def\textqed{\unskip\nobreak\penalty50
        \hskip2em\hbox{}\nobreak\hfil\qedsymbol
        \parfillskip=0pt \finalhyphendemerits=0}
 \def\dispqed{\rlap{\qquad\qedsymbol}}
 
 \opn\dis{dis}
 \def\pnt{{\raise0.5mm\hbox{\large\bf.}}}
 
 \opn\Lex{Lex}

\begin{document}
%\linenumbers
\title {Projective dimension and  regularity  of  edge ideal of some weighted oriented graphs}

\author {Guangjun Zhu$^{^*}$, Li Xu, Hong Wang and Zhongming Tang}

\address{Authors¡¯ address:  School of Mathematical Sciences, Soochow
University, Suzhou 215006, P.R. China}
\email{zhuguangjun@suda.edu.cn(Guangjun Zhu), 1240470845@qq.com(Li Xu), 651634806@qq.com(Hong Wang), zmtang@suda.edu.cn(Zhongming Tang).}

\dedicatory{ }

\begin{abstract}
In this paper we provide some exact formulas for the projective dimension and the regularity
of edge ideals associated to vertex weighted rooted forests and  oriented  cycles. As some consequences, we give  some exact formulas for the depth
of these ideals.
\end{abstract}

\thanks{* Corresponding author}

\subjclass[2010]{13D02, 13F55, 13C15, 13D99}
%		13H10   	Special types (Cohen-Macaulay, Gorenstein, Buchsbaum, etc.)
%		13D02   	Syzygies, resolutions, complexes
%		05E40   	Combinatorial aspects of commutative algebra
%		16S36   	Ordinary and skew polynomial rings and semigroup rings

%		14M25   	Toric varieties, Newton polyhedra [See also 52B20]
%		13A02   	Graded rings
%		13F20   	Polynomial rings and ideals; rings of integer-valued polynomials
%		13A18   	Valuations and their generalizations
%		06A11   	Algebraic aspects of posets

\keywords{projective dimension, regularity, weighted rooted forest, weighted oriented  cycle, edge ideal}

\maketitle

\setcounter{tocdepth}{1}
%\tableofcontents

\section{Introduction}

\hspace{3mm} A {\em directed graph} $D$ consists of a finite set $V(D)$ of vertices, together
with a prescribed collection $E(D)$ of ordered pairs of distinct points called edges or
arrows.  If $\{u,v\}\in E(D)$ is an edge, we write $uv$ for $\{u,v\}$, which is denoted to be the directed edge
where the direction is from $u$ to $v$ and $u$ (resp. $v$) is called the {\em starting}  point (resp. the {\em ending} point).
An {\em  oriented} graph is a  directed  graph  having no bidirected edges (i.e. each pair of vertices is joined by a single edge having a unique direction). In other words an oriented graph $D$ is a simple graph $G$ together with an orientation
of its edges. We call $G$ the underlying graph of $D$.

An oriented graph having no multiple edges (edges with same starting point and ending point) or loops (edges that connect vertices to themselves) is called a {\em simple oriented graph}.  An {\em oriented} cycle is a cycle graph in which all edges are oriented in  clockwise  or in counterclockwise.
An oriented acyclic graph is a simple directed graph without oriented cycles.
An {\em  oriented tree} or polytree is a  oriented acyclic graph formed by orienting the edges of undirected acyclic graphs.
A  {\em rooted tree} is an oriented tree in which all edges  are oriented  either away from or
towards the root.  Unless specifically stated, a rooted tree in this article
 is an oriented tree in which all edges  are oriented away from  the root.
A {\em leaf} of a  tree is a vertex adjacent to only one other vertex.
An {\em oriented forest} is a disjoint union of oriented trees. A {\em rooted forest} is a disjoint union of rooted trees.

 A vertex-weighted oriented graph is a triplet $D=(V(D), E(D),w)$, where $V(D)$ is the  vertex set,
$E(D)$ is the edge set and $w$ is a weight function $w: V(D)\rightarrow \mathbb{N}^{+}$, here $N^{+}=\{1,2,\ldots\}$.
Some times for short we denote the vertex set $V(D)$ and edge set $E(D)$
by $V$ and $E$ respectively.
The weight of $x_i\in V$ is $w(x_i)$, denoted by $w_i$.

The edge ideal of a vertex-weighted directed graph was first introduced by Gimenez et al \cite{GBSVV}. Let $D=(V,E,w)$ be a  vertex-weighted oriented graph with the vertex set $V=\{x_{1},\ldots,x_{n}\}$. We consider the polynomial ring $S=k[x_{1},\dots, x_{n}]$ in $n$ variables over a field $k$. The edge ideal of $D$,  denoted by $I(D)$, is the ideal of $S$ given by
$$I(D)=(x_ix_j^{w_j}\mid  x_ix_j\in E).$$

If  $w_j=1$ for all $j$, then $I(D)$ is the edge ideal of underlying graph $G$ of $D$.
It has been extensively studied in the literature \cite{HT3,MV,SVV,T,V1}.
Especially the study of algebraic invariants corresponding to their minimal free resolutions has become popular(see
 \cite{AF1,AF2,BHTN,BHO,DHS,F, HT2,J,KM,Z1,Z2}).

In this article, we focus on  algebraic properties  corresponding to projective dimension and regularity of the edge ideals of some weighted oriented graphs.

Edge ideals of edge-weighted graphs were introduced and studied by Paulsen
and Sather-Wagstaff \cite{PS}.
In this work we consider edge ideals of vertex-weighted oriented graph (See \cite{GBSVV,PRT}).
In what follows by a weighted  oriented graph we shall always mean a vertex-weighted oriented graph.
Edge ideals of weighted oriented graphs arose in the theory of Reed-Muller codes as initial ideals of
vanishing ideals of projective spaces over finite fields \cite{MPV}.

This paper is organized as follows. In the next section, we recall several definitions and
terminology which we need later. In Secetions 3 and 4, using  the approach of a Betti splitting and  polarization, we derive some exact formulas for the projective dimension and regularity of the edges of weighted rooted forests and oriented  cycles. The results are as follows:

\begin{Theorem}
Let $D(V(D), E(D),w)$ be a weighted oriented star graph
such that  $w(x)\geq 2$  for any vertex $x$. If $E(D)$ is one of the following cases $\{x_1x_2,x_1x_3\ldots,x_1x_n\}$, $\{x_2x_1,x_3x_1\ldots,x_nx_1\}$ and $\{x_1x_2,x_2x_3\ldots,x_2x_n\}$. Then
$$\mbox{pd}\,(I(D))=|E(D)|-1,\ \ \mbox{and}\ \
 \mbox{reg}\,(I(D))=\sum\limits_{x\in V(D)}w(x)-|E(D)|+1.$$
\end{Theorem}

\begin{Theorem}
Let $D=(V(D),E(D),w)$ be a weighted rooted forest such that  $w(x)\geq 2$   for any vertex $x$.
Then $\mbox{pd}\,(I(D))=|E(D)|-1$.
\end{Theorem}

\begin{Theorem}
Let $D=(V(D),E(D),w)$ be a weighted rooted forest such that  $w(x)\geq 2$  for any vertex $x$. Then $\mbox{reg}\,(I(D))=\sum\limits_{x\in V(D)}w(x)-|E(D)|+1$.
\end{Theorem}

\begin{Theorem}
Let $D=(V(D),E(D),w)$ be a weighted oriented  cycle such that  $w(x)\geq 2$   for any  vertex $x$. Then
$$\mbox{pd}\,(I(D))=|E(D)|-1 \ \ \mbox{and}\ \
 \mbox{reg}\,(I(D))=\sum\limits_{x\in V(D)}w(x)-|E(D)|+1.$$
\end{Theorem}

\medskip
For all unexplained terminology and additional information, we refer to \cite{JG} (for the theory
of digraphs), \cite{H} (for graph theory), and \cite{HH} (for the theory of edge ideals of graphs and
monomial ideals).

\medskip

\section{Preliminaries }

In this section, we gather together the needed  definitions and basic facts, which will
be used throughout this paper. However, for more details, we refer the reader to \cite{HH,FHT,HT1,Z1}.

\medskip
For any homogeneous ideal $I$ of the polynomial ring  $S=k[x_{1},\dots,x_{n}]$, there exists a {\em graded
minimal finite free resolution}

\vspace{3mm}
$0\rightarrow \bigoplus\limits_{j}S(-j)^{\beta_{p,j}(M)}\rightarrow \bigoplus\limits_{j}S(-j)^{\beta_{p-1,j}(M)}\rightarrow \cdots\rightarrow \bigoplus\limits_{j}S(-j)^{\beta_{0,j}(M)}\rightarrow I\rightarrow 0,$
where the maps are exact, $p\leq n$, and $S(-j)$ is the $S$-module obtained by shifting
the degrees of $S$ by $j$. The number
$\beta_{i,j}(I)$, the $(i,j)$-th graded Betti number of $I$,  is
an invariant of $I$ that equals the number of minimal generators of degree $j$ in the
$i$th syzygy module of $I$.
Of particular interest are the following invariants which measure the ¡°size¡± of the minimal graded
free resolution of $I$.
The projective dimension of $I$, denoted pd\,$(I)$, is defined to be
$$\mbox{pd}\,(I):=\mbox{max}\{i\ |\ \beta_{i,j}(I)\neq 0\}.$$
The regularity of $I$, denoted $\mbox{reg}\,(I)$, is defined by
$$\mbox{reg}\,(I):=\mbox{max}\{j-i\ |\ \beta_{i,j}(I)\neq 0\}.$$

\vspace{3mm}We now derive some formulas for $\mbox{pd}\,(I)$  and $\mbox{reg}\,(I)$ in some special cases by using some
tools developed in \cite{FHT}.

\begin{Definition} \label{bettispliting}Let $I$  be a monomial ideal, and suppose that there exist  monomial
ideals $J$ and $K$ such that $\mathcal{G}(I)$ is the disjoint union of $\mathcal{G}(J)$ and $\mathcal{G}(K)$. Then $I=J+K$
is a {\em Betti splitting} if
$$\beta_{i,j}(I)=\beta_{i,j}(J)+\beta_{i,j}(K)+\beta_{i-1,j}(J\cap K)\hspace{2mm}\mbox{for all}\hspace{2mm}i,j\geq 0,$$
where $\beta_{i-1,j}(J\cap K)=0\hspace{2mm}  \mbox{if}\hspace{2mm} i=0$.
\end{Definition}

\vspace{3mm}This formula was first obtained for the total Betti numbers by
Eliahou and Kervaire \cite{EK} and extended to the graded case by Fatabbi \cite{F}.
In  \cite{FHT}, the authors describe some sufficient conditions for an
ideal $I$ to have a Betti splitting. We shall require the following such condition.

\begin{Theorem}\label{Thm1}(\cite[Corollary 2.7]{FHT}).
Suppose that $I=J+K$ where $\mathcal{G}(J)$ contains all
the generators of $I$ divisible by some variable $x_{i}$ and $\mathcal{G}(K)$ is a nonempty set containing
the remaining generators of $I$. If $J$ has a linear resolution, then $I=J+K$ is a Betti
splitting.
\end{Theorem}

\medskip
When $I$ is a Betti
splitting ideal,  Definition \ref{bettispliting} implies the following results:
\begin{Corollary} \label{cor1}
If $I=J+K$ is a Betti splitting ideal, then
\begin{itemize}
 \item[(1)]$\mbox{reg}\,(I)=\mbox{max}\{\mbox{reg}\,(J),\mbox{reg}\,(K),\mbox{reg}\,(J\cap K)-1\}$,
 \item[(2)] $\mbox{pd}\,(I)=\mbox{max}\{\mbox{pd}\,(J),\mbox{pd}\,(K),\mbox{pd}\,(J\cap K)+1\}$.
\end{itemize}
\end{Corollary}

\medskip
We need the following Lemmas:
\begin{Lemma}
\label{lem1}(\cite[Lemma 3.1]{Z1})
Let $S_{1}=k[x_{1},\dots,x_{m}]$ and $S_{2}=k[x_{m+1},\dots,x_{n}]$ be two polynomial rings, $I\subseteq S_{1}$ and
$J\subseteq S_{2}$ be two nonzero homogeneous  ideals. Then
\begin{itemize}
\item[(1)]$\mbox{pd}\,(I+J)=\mbox{pd}\,(I)+\mbox{pd}\,(J)+1$,
\item[(2)]$\mbox{reg}\,(I+J)=\mbox{reg}\,(I)+\mbox{reg}\,(J)-1$,
\item[(3)]$\mbox{reg}\,(IJ)=\mbox{reg}\,(I)+\mbox{reg}\,(J)$.
\end{itemize}
\end{Lemma}

\medskip
 Let $\mathcal{G}(I)$ denote the minimal set of generators of a monomial ideal $I\subset S=k[x_{1},\dots, x_{n}]$
 and let $u\in S$ be a monomial, we set $\mbox{supp}(u)=\{x_i: x_i|u\}$. If $\mathcal{G}(I)=\{u_1,\ldots,u_m\}$, we set $\mbox{supp}(I)=\bigcup\limits_{i=1}^{m}\mbox{supp}(u_i)$. The following lemma is well known.

\begin{Lemma}
\label{lem2} Let  $I, J=(u)$ be two monomial ideals  such that $\mbox{supp}(u)\cap \mbox{supp}(I)=\emptyset$, where $u$  is a monomial of degree $m$. Then
\begin{itemize}
\item[(1)] $\mbox{reg}\,(J)=m$,
\item[(2)]$\mbox{pd}\,(uI)=\mbox{pd}\,(I)$,
\item[(3)]$\mbox{reg}\,(uI)=\mbox{reg}\,(I)+m$.
\end{itemize}
\end{Lemma}

\medskip
\begin{Definition} \label{polarization}
Suppose that $u=x_1^{a_1}\cdots x_n^{a_n}$ is a monomial in $S$. Then we define the {\it polarization} of $u$ to be
the squarefree monomial $$\mathcal{P}(u)=x_{11}x_{12}\cdots x_{1a_1} x_{21}\cdots x_{2a_2}\cdots x_{n1}\cdots x_{na_n}$$
in the polynomial ring $S^{\mathcal{P}}=k[x_{ij}\mid 1\leq i\leq n, 1\leq j\leq a_i]$.
If $I\subset S$ is a monomial ideal with $\mathcal{G}(I)=\{u_1,\ldots,u_m\}$,  the  {\it polarization}
of $I$,  denoted by $I^{\mathcal{P}}$, is defined as:
$$I^{\mathcal{P}}=(\mathcal{P}(u_1),\ldots,\mathcal{P}(u_m)),$$
which is a squarefree monomial ideal in a polynomial ring $S^{\mathcal{P}}$.
\end{Definition}

\medskip
Here is two examples of how polarization works.
\begin{Example} \label{example1} Let $I=(x_1^{2}x_2^{3},x_2^{4}x_3,x_3x_4^{2},x_4^{2}x_5)$ be a monomial ideal, then the polarization of $I$ is the ideal  $I^{\mathcal{P}}=(x_{11}x_{12}x_{21}x_{22}x_{23},x_{21}x_{22}x_{23}x_{24}x_{31},
x_{31}x_{41}x_{42},x_{41}x_{42}x_{51})$.
\end{Example}

\begin{Example} \label{example2} Let $I(D)=(x_1x_2^{3},x_2x_3,x_3x_4^{2},x_4x_5^{5})$ be the edge ideal of a weighted rooted tree $D$, then the polarization of $I(D)$ is the ideal  $I(D)^{\mathcal{P}}=(x_{11}x_{21}x_{22}x_{23},x_{21}x_{31},\\
x_{31}x_{41}x_{42},x_{41}x_{51}x_{52}x_{53}x_{54}x_{55})$.
\end{Example}

\medskip
A monomial ideal $I$ and its polarization $I^{\mathcal{P}}$ share many homological and
algebraic properties. Thus, by polarization, many questions concerning monomial
ideals can be reduced to squarefree monomial ideals. The following is a very useful property of polarization which play an essential role
throughout the paper.

\begin{Lemma}
\label{lem3}(\cite[Corollary 1.6.3]{HH}) Let $I\subset S$ be a monomial ideal and $I^{\mathcal{P}}\subset S^{\mathcal{P}}$ its polarization.
Then
\begin{itemize}
\item[(1)] $\beta_{ij}(I)=\beta_{ij}(I^{\mathcal{P}})$ for all $i$ and $j$,
\item[(2)] $\mbox{pd}\,(I)=\mbox{pd}\,(I^{\mathcal{P}})$ and $\mbox{reg}\,(I)=\mbox{reg}\,(I^{\mathcal{P}})$.
\end{itemize}
\end{Lemma}

\medskip

\medskip

\section{Projective dimensions and  regularities of  edge ideals of rooted forests}

\hspace{3mm}In this section, by using the approach of a Betti splitting and  polarization, we will provide some  formulas for computing the projective dimensions and  regularities of   edge ideals of some weighted oriented star graphs and rooted forests. As some consequences, we  also give some exact formulas for the depths of edge ideals of oriented forests.

\begin{Theorem}\label{thm2}
Let $D(V(D), E(D),w)$ be a weighted oriented star graph
such that  $w(x)\geq 2$  for any vertex $x$. If $E(D)$ is one of the following cases $\{x_1x_2,x_1x_3\ldots,x_1x_n\}$, $\{x_2x_1,x_3x_1\ldots,x_nx_1\}$ and $\{x_1x_2,x_2x_3\ldots,x_2x_n\}$. Then
$$\mbox{pd}\,(I(D))=|E(D)|-1,\ \ \mbox{and}\ \
 \mbox{reg}\,(I(D))=\sum\limits_{x\in V(D)}w(x)-|E(D)|+1.$$
\end{Theorem}
\begin{proof} Let us assume that $V(D)=\{x_1,\ldots,x_n\}$ and $w_i=w(x_i)$ for any $x_i$.

(1) If $E(D)=\{x_1x_2,x_1x_3\ldots,x_1x_n\}$, then
$I(D)=(x_1x_2^{w_2},\ldots,x_1x_n^{w_n})$. The conclusions  follows from  Lemmas \ref{lem1} and \ref{lem2}.

(2) If $E(D)=\{x_2x_1,x_3x_1\ldots,x_nx_1\}$, then  by similar arguments as above, the desired follows.

(3) If $E(D)=\{x_1x_2,x_2x_3\ldots,x_2x_n\}$, then $$I(D)=(x_1x_2^{w_2},x_2x_3^{w_3}\ldots,x_2x_n^{w_n})=J+K,$$
 where $J=(x_1x_2^{w_2})$ has a linear resolution, $K=(x_2x_3^{w_3}\ldots,x_2x_n^{w_n})=x_2L$ and
 $J\cap K=JL$ where $L=(x_3^{w_3}\ldots,x_n^{w_n})$. Thus $\mbox{pd}\,(K)=\mbox{pd}\,(L)=n-3$,  $\mbox{reg}\,(J)=w_2+1$, $\mbox{reg}\,(L)=\sum\limits_{i=3}^{n}w_i-(n-3)$. by Lemmas \ref{lem1} and \ref{lem2}.
It follows that from   Corollary \ref{cor1}, Lemma \ref{lem2} \begin{eqnarray*}\mbox{pd}\,(I(D))&=&\mbox{max}\{\mbox{pd}\,(J),\mbox{pd}\,(K),\mbox{pd}\,(J\cap K)+1\}\\
&=&\mbox{max}\{0,n-3,n-3+1\}\\
&=&n-2=|E(D)|-1.
\end{eqnarray*}
and
\begin{eqnarray*}\mbox{reg}\,(I(D))&=&\mbox{max}\{\mbox{reg}\,(J),\mbox{reg}\,(K),\mbox{reg}\,(J\cap K)-1\}\\
&=&\mbox{max}\{w_2+1,\mbox{reg}\,(L)+1,\mbox{reg}\,(J)+\mbox{reg}\,(L)-1\}\\
&=&\mbox{reg}\,(J)+\mbox{reg}\,(L)-1=(w_2+1)+\sum\limits_{i=3}^{n}w_i-(n-3)-1\\
&=&\sum\limits_{i=2}^{n}w_i-(n-3)=\sum\limits_{i=1}^{n}w_i-(n-1)+1\\
&=&\sum\limits_{i=1}^{n}w_i-|E(D)|+1.
\end{eqnarray*}
\end{proof}

\vspace{3mm}As a consequence  of the above theorem, we have
\begin{Corollary} \label{cor2}
Let $D$ be aweighted oriented star graph  as in Theorem \ref{Thm2}. Then $$\mbox{depth}\,(I(D))=2.$$
\end{Corollary}
\begin{proof} By Auslander-Buchsbaum formula (see Theorem 1.3.3 of \cite{BH}), it follows that
$$\mbox{depth}\,(I(D))=(|E(D)|+1)-\mbox{pd}\,(I(D))=2.$$
\end{proof}

\medskip Let  $D=(V(D), E(D),w)$ be a vertex-weighted oriented graph. For $T\subset V$, we define
the {\em induced vertex-weighted  subgraph} $H=(V(H), E(H),w)$ of $D$  to be the vertex-weighted oriented graph
such that $V(H)=T$, $uv\in E(H)$  if and only if $uv\in E(D)$  and   for any $u\in V(H)$, its  weight in $H$ equals to the weight of $u$ in $D$.

\medskip
We now state and prove two main theorems of this section.

\begin{Theorem}\label{Thm3}
Let $D=(V(D),E(D),w)$ be a weighted rooted forest such that  $w(x)\geq 2$   for any vertex $x$.
Then $\mbox{pd}\,(I(D))=|E(D)|-1$.
\end{Theorem}
\begin{proof} Let $D_1,\ldots,D_t$ be all  connected components of $D$, then for each $D_i$, it is a rooted tree.
Lemma \ref{lem1} (1) implies $$\mbox{pd}\,(I(D))=\mbox{pd}\,(\sum\limits_{i=1}^{t}I(D_i))=\sum\limits_{i=1}^{t}\mbox{pd}\,(I(D_i))+t-1.$$
It is enough to prove that $\mbox{pd}\,(I(D_i))=|E(D_i)|-1$ for  $i=1,\ldots,t$. Hence
we may assume that $D$ is a rooted tree with  the vertex set $V=\{x_1,\ldots, x_n\}$. We apply induction on  $|E(D)|$.

The case $D$ is a weighted oriented star graph follows from Theorem \ref{thm2}.
 Let $x_n$ be a leaf of $D$ which  adjacent to $x_{n-1}$ and $w_i=w(x_i)$ for any $i$, then
$$I(D)=(x_{n-1}x_n^{w_n})+I(D\setminus \{x_n\}),$$
where $D\setminus \{x_n\}$  is the subgraph of $D$ with vertex $x_n$ and edge $x_{n-1}x_n$ removed. Thus
$D\setminus \{x_n\}$ is a rooted subtree  with $|E(D)|-1$ edges. By the inductive hypothesis, we have $\mbox{pd}\,(I(D\setminus \{x_n\}))=|E(D\setminus \{x_n\})|-1=|E(D)|-2$.

Let  $I(D)^{\mathcal{P}}$ be the polarization of $I(D)$, then
$$
I(D)^{\mathcal{P}}=(x_{n-1,1}\prod\limits_{j=1}^{w_n} x_{nj})+I(D\setminus \{x_n\})^{\mathcal{P}}.
$$
Set $J=(x_{n-1,1}\prod\limits_{j=1}^{w_n} x_{nj})$ and $K=I(D\setminus \{x_n\})^{\mathcal{P}}$, then $J$ has a linear resolution, hence $I(D)^{\mathcal{P}}=J+K$ is a Betti splitting by Theorem
\ref{Thm1} and
\begin{eqnarray*}
J\cap K\!\!\!&=&\!\!\!(x_{n-1,1}\!\!\prod\limits_{j=1}^{w_n} \!x_{nj})((\!\!\prod\limits_{j=1}^{w_{m+1}}\!\!x_{m+1,j},\ldots,\!\!\prod\limits_{j=1}^{w_{n-2}} \!\!x_{n-2,j},x_{m1}\!\!\prod\limits_{j=2}^{w_{n-1}}\!\!x_{n-1,j})\!+\!I(D\setminus \{x_{n-1},x_n\})^{\mathcal{P}})\\
\!\!&=&\!\!\!J((\!\!\prod\limits_{j=1}^{w_{m+1}}\!\!x_{m+1,j},\ldots,\!\!\prod\limits_{j=1}^{w_{n-2}} \!\!x_{n-2,j},x_{m1}\!\!\prod\limits_{j=2}^{w_{n-1}}\!\!x_{n-1,j})+I(D\setminus \{x_{n-1},x_n\})^{\mathcal{P}}),
\end{eqnarray*}
where $x_mx_{n-1},x_{n-1}x_{m+1},\ldots,x_{n-1}x_{n-2}\in E(D\setminus \{x_n\})$ are all edges adjacent to $x_{n-1}$ in $D\setminus \{x_n\}$ and $D\setminus \{x_{n-1},x_n\}$ is a rooted forest with the vertices $D\setminus \{x_{n-1},x_n\}$ and  edges  incident to $x_{n-1}$ or $x_n$ removed.

Let $L=(\prod\limits_{j=1}^{w_{m+1}}x_{m+1,j},\ldots,\prod\limits_{j=1}^{w_{n-2}}x_{n-2,j}, x_{m1}\prod\limits_{j=2}^{w_{n-1}}x_{n-1,j})+I(D\setminus \{x_{n-1},x_n\})^{\mathcal{P}}$, then
$L$ is  the polarization of the edge ideal of a rooted forest $H$ having $n-m-1$  connected components, each component is added  another leaf with weight $w_i-1$ to its root $x_i$ for $i=m+1,\ldots,n-2$. This implies $|E(H)|=|E(D)|-1$.
It follows that $\mbox{pd}\,(L)=|E(H)|-1=|E(D)|-2$ by the inductive hypothesis.

 Again using Lemma \ref{lem3} (2), Corollary \ref{cor1} (2) and Lemma \ref{lem2} (2), we obtain
\begin{eqnarray*}
\mbox{pd}\,(I(D))&=&\mbox{pd}\,(I(D)^{\mathcal{P}})=\mbox{max}\{\mbox{pd}\,(J), \mbox{pd}\,(K), \mbox{pd}\,(J\cap K)+1\}\\
&=&\mbox{max}\{0, \mbox{pd}\,(I(D\setminus \{x_n\})^{\mathcal{P}}), \mbox{pd}\,(L)+1\}\\
&=&\mbox{max}\{0, \mbox{pd}\,(I(D\setminus \{x_n\})), \mbox{pd}\,(L)+1\}\\
&=&\mbox{max}\{0, |E(D)|-2, |E(D)|-1\}\\
&=&|E(D)|-1.
\end{eqnarray*}
The proof is complete.
\end{proof}

\vspace{3mm}As a consequence  of the above theorem, we have
\begin{Corollary} \label{cor3}
Let $D$ be a weighted rooted forest  as in Theorem \ref{Thm3}. Then $$\mbox{depth}\,(I(D))=2.$$
\end{Corollary}
\begin{proof} By Auslander-Buchsbaum formula (see Theorem 1.3.3 of \cite{BH}), it follows that
$$\mbox{depth}\,(I(D))=(|E(D)|+1)-\mbox{pd}\,(I(D))=2.$$
\end{proof}

\medskip
\begin{Theorem}\label{Thm4}
Let $D=(V(D),E(D),w)$ be a weighted rooted forest such that  $w(x)\geq 2$  for any vertex $x$. Then $\mbox{reg}\,(I(D))=\sum\limits_{x\in V(D)}w(x)-|E(D)|+1$.
\end{Theorem}
\begin{proof}  Let $D_1,\ldots,D_t$ be all  connected components of $D$, then for each $D_i$, it is a rooted tree.
By Lemma \ref{lem1} (2), we have $$\mbox{reg}\,(I(D))=\mbox{reg}\,(\sum\limits_{i=1}^{t}I(D_i))=\sum\limits_{i=1}^{t}\mbox{reg}\,(I(D_i))-(t-1).$$
It is enough to prove that $\mbox{reg}\,(I(D_i))=\sum\limits_{x\in V(D_i)}w(x)-|E(D_i)|+1$ for  $i=1,\ldots,t$. Hence
we may assume that $D$ is a rooted tree with  the vertex set $V=\{x_1,\ldots, x_n\}$. We apply induction on $|E(D)|$.

The case $D$ is a weighted oriented star graph follows from Theorem \ref{thm2}.
Let $x_n$ be  a leaf of $D$ which  adjacent to $x_{n-1}$, then by the proof of Theorem \ref{Thm3}, we have
$$I(D)=(x_{n-1}x_n^{w_n})+I(D\setminus \{x_n\}),\hspace{3mm}
I(D)^{\mathcal{P}}=(x_{n-1,1}\prod\limits_{j=1}^{w_n} x_{nj})+I(D\setminus \{x_n\})^{\mathcal{P}}
$$
and
$I(D)^{\mathcal{P}}$ is a Betti splitting ideal with splitting $I(D)^{\mathcal{P}}=J+K$,
where  $J=(x_{n-1,1}\prod\limits_{j=1}^{w_n} x_{nj})$, $K=I(D\setminus \{x_n\})^{\mathcal{P}}$, $J\cap K=JL$
and $$
L=(\prod\limits_{j=1}^{w_{m+1}}x_{m+1,j},\ldots,\prod\limits_{j=1}^{w_{n-2}} x_{n-2,j},x_{m1}\prod\limits_{j=2}^{w_{n-1}}x_{n-1,j})+I(D\setminus \{x_{n-1},x_n\})^{\mathcal{P}},
$$
where $D\setminus \{x_n\}$  is the subgraph of $D$ with vertex $x_n$ and edge $x_{n-1}x_n$ removed.
$x_mx_{n-1},x_{n-1}x_{m+1},\ldots,x_{n-1}x_{n-2}\in E(D\setminus \{x_n\})$ are all edges adjacent to $x_{n-1}$ in $D\setminus \{x_n\}$, $D\setminus \{x_{n-1},x_n\}$ is a rooted forest with the vertices $D\setminus \{x_{n-1},x_n\}$ and  edges  incident to $x_{n-1}$ or $x_n$ removed, $L$ is  the polarization of the edge ideal of a rooted forest $H$ having $(n-m-1)$  connected components, each component is added  another leaf with weight $w_i-1$ to its root $x_i$ for $i=m+1,\ldots,n-2$.
Thus
$D\setminus \{x_n\}$ is a rooted tree with $|E(D)|-1$ edges. By the inductive hypothesis, we have
\begin{eqnarray*}\mbox{reg}\,(I(D\setminus \{x_n\}))&=&\sum\limits_{i=1}^{n-1}w_i-|E(D\setminus \{x_n\})|+1\\
&=&\sum\limits_{i=1}^{n-1}w_i-|E(D)|+2,
\end{eqnarray*}
and
\begin{eqnarray*}\mbox{reg}\,(L)&=&\sum\limits_{x\in V(H)}w(x)-|E(H)|+(n-m-2)
-(n-m-2)\\
&=&\sum\limits_{x\in V(H)}w(x)-|E(H)|=\sum\limits_{x\in V(H)}w(x)-(|E(D)|-1)\\
&=&\sum\limits_{x\in V(H)}w(x)-|E(D)|+1.
\end{eqnarray*}
Lemma \ref{lem2} implies
\begin{eqnarray*}\mbox{reg}\,(J\cap K)&=&\mbox{reg}\,(J)+\mbox{reg}\,(K)\\
&=&(w_n+1)+\sum\limits_{x\in V(H)}w(x)-|E(D)|+1\\
&=&\sum\limits_{i=1}^{n}w_i-|E(D)|+2.
\end{eqnarray*}

It follows that
\begin{eqnarray*}\mbox{reg}\,(I(D))&=&\mbox{reg}\,(I(D)^{\mathcal{P}})=\mbox{max}\{\mbox{reg}\,(J),\mbox{reg}\,(K),\mbox{reg}\,(J\cap K)-1\}\\
&=&\mbox{max}\{w_n+1,\sum\limits_{i=1}^{n-1}w_i-|E(D)|+2,\sum\limits_{i=1}^{n}w_i-|E(D)|+1\}\\
&=&\sum\limits_{i=1}^{n}w_i-|E(D)|+1.
\end{eqnarray*}
The proof is complete.
\end{proof}

\medskip
\section{Projective dimensions and  regularities of  edge ideals of  oriented  cycles}

In this section, we will provide some  formulas for the projective dimensions and  regularities of   edge ideals of some oriented cycles. As some consequences, we  also give some exact formulas for the depth
of edge ideals of oriented cycles.

\begin{Theorem}\label{Thm4}
Let $D=(V(D),E(D),w)$ be a weighted oriented  cycle such that  $w(x)\geq 2$   for any  vertex $x$. Then
$$\mbox{pd}\,(I(D))=|E(D)|-1 \ \ \mbox{and}\ \
 \mbox{reg}\,(I(D))=\sum\limits_{x\in V(D)}w(x)-|E(D)|+1.$$
\end{Theorem}
\begin{proof}  Let $V=\{x_1,\ldots, x_n\}$ and $w_i=w(x_i)$, then
  \begin{eqnarray*}I(D)&=&(x_1x_2^{w_2},\ldots,x_{n-1}x_n^{w_n},x_nx_1^{w_1}),\\
I(D)^{\mathcal{P}}&=&(x_{11}\!\!\prod\limits_{j=1}^{w_2}x_{2j},\ldots,x_{n-1,1}\!\!\prod\limits_{j=1}^{w_n}x_{nj},
x_{n1}\!\!\prod\limits_{j=1}^{w_1}x_{1j}).
\end{eqnarray*}
Set $L_1=I(D)^{\mathcal{P}}$, $J_1=(x_{n1}\!\!\prod\limits_{j=1}^{w_1}x_{1j})$, $K_1=(x_{11}\!\!\prod\limits_{j=1}^{w_2}x_{2j},\ldots,x_{n-1,1}\!\!\prod\limits_{j=1}^{w_n}x_{nj})$,
$J_i=(\prod\limits_{j=1}^{w_i}x_{ij})$, $K_i=(x_{i1}\!\!\prod\limits_{j=1}^{w_{i+1}}x_{i+1,j},\ldots,x_{n-1,1}\!\!\prod\limits_{j=2}^{w_n}x_{nj})$,
$L_i=(\prod\limits_{j=1}^{w_i}x_{ij},x_{i1}\!\!\prod\limits_{j=1}^{w_{i+1}}x_{i+1,j},\ldots,x_{n-1,1}\!\!\prod\limits_{j=2}^{w_n}x_{nj})$
for any $2\leq i\leq n-1$ and $L_n=(\prod\limits_{j=2}^{w_n}x_{nj})$.
Notice that all $J_i$ have  linear resolutions for  $1\leq i\leq n-1$, it follows that
$L_i=J_i+K_i$ is a Betti splitting.
Also notice that $J_i\cap K_i=J_iL_{i+1}$ and the fact the variables that appear in $J_i$ and $L_{i+1}$ are different, by Lemma \ref{lem2}, Corollary \ref{cor1},
we obtain,  for $1\leq i\leq n-1$,
\begin{eqnarray*}
\mbox{pd}\,(J_{i}\cap K_{i})&=&\mbox{pd}\,(L_{i+1})=\mbox{max}\{\mbox{pd}\,(J_{i+1}),\mbox{pd}\,(K_{i+1}), \mbox{pd}\,(J_{i+1}\cap K_{i+1})+1\},\\
\mbox{reg}\,(J_{i}\cap K_{i})&=&\mbox{reg}\,(J_{i}L_{i+1})=\mbox{reg}\,(J_{i})+\mbox{reg}\,(L_{i+1}),\\
\mbox{reg}\,(L_{i+1})&=&\mbox{max}\{\mbox{reg}\,(J_{i+1}),\mbox{reg}\,(K_{i+1}), \mbox{reg}\,(J_{i+1}\cap K_{i+1})-1\}. \hspace{3cm} (1)
\end{eqnarray*}
Since $J_{i}$, $K_{n-1}$ and $J_{n-1}\cap K_{n-1}$  are principal  ideals, $\mbox{pd}\,(J_{i})=\mbox{pd}\,(J_{n-1}\cap K_{n-1})=\mbox{pd}\,(K_{n-1})=0$ for $1\leq i\leq n-1$.
By repeated use of  the above equalities (1) and  induction on $n$ and $i$, we can obtain that  $\mbox{pd}\,(K_{i})=n-i-1$, $\mbox{pd}\,(L_{i})=n-i$, $\mbox{reg}\,(K_{i})=\mbox{reg}\,(L_{i+1})+1$, $\mbox{reg}\,(L_{i})=\sum\limits_{j=i}^{n}w_j-(n-i+1)$.
It follows that $\mbox{pd}\,(K_{1})=n-2$, $\mbox{pd}\,(J_{1}\cap K_{1})=\mbox{pd}\,(L_{2})=n-2$ and $\mbox{reg}\,(K_{1})=\sum\limits_{j=2}^{n}w_j-(n-1)$ and
$\mbox{reg}\,(J_{1}\cap K_{1})=\sum\limits_{j=1}^{n}w_j-(n-2)$.
Thus,
\begin{eqnarray*}\mbox{pd}\,(L_{1})&=&\mbox{max}\{\mbox{pd}\,(J_{1}),\mbox{pd}\,(K_{1}), \mbox{pd}\,(J_{1}\cap K_{1})+1\}\\
&=&\mbox{max}\{0,n-2,n-2+1\}=n-1,
\end{eqnarray*}
\begin{eqnarray*}\mbox{reg}\,(L_{1})&=&\mbox{max}\{\mbox{reg}\,(J_{1}),\mbox{reg}\,(K_{1}), \mbox{reg}\,(J_{1}\cap K_{1})-1\}\\
&=&\mbox{max}\{w_1+1,\sum\limits_{j=2}^{n}w_j-(n-1),\sum\limits_{j=1}^{n}w_j-(n-2)-1\}\\
&=&\sum\limits_{j=1}^{n}w_j-(n-1).
\end{eqnarray*}
This concludes the proof of the theorem.
\end{proof}

\medskip
An immediate consequence of the above theorem is the following corollary.
\begin{Corollary} \label{cor4}
Let $D=(V(D),E(D),w)$ be a weighted oriented  cycle such that  $w(x)\geq 2$   for any  vertex $x$. Then  $\mbox{depth}\,(I(D))=1$.
\end{Corollary}

\begin{proof} By Auslander-Buchsbaum formula (see Theorem 1.3.3 of \cite{BH}), it follows that
$$\mbox{depth}\,(I(D))=|E(D)|-\mbox{pd}\,(I(D))=1.$$
\end{proof}

\medskip

\hspace{-6mm} {\bf Acknowledgments}

 \vspace{3mm}
\hspace{-6mm}  This research is supported by the National Natural Science Foundation of China (No.11271275 and No.11471234) and  by foundation of the Priority Academic Program Development of Jiangsu Higher Education Institutions.

\end{document}